\documentclass[12pt, twoside, leqno]{article}
\usepackage{amsmath,amsthm}
\usepackage{amssymb,latexsym}
\usepackage{enumerate}

\newtheorem{theorem}{Theorem}[section]
\newtheorem{corollary}[theorem]{Corollary}
\newtheorem{lemma}[theorem]{Lemma}

\theoremstyle{definition}

\newtheorem{rem}[theorem]{Remark}

\def \N {{\mathbb N}}

\numberwithin{equation}{section}

\begin{document}


\baselineskip=17pt


\title{Egyptian Fractions with odd denominators}

\author{
Christian Elsholtz\footnote{
Institut f\"ur Analysis und Zahlentheorie,
Graz University of Technology,
Kopernikusgasse 24,
A-8010 Graz, Austria.
E-mail: elsholtz@math.tugraz.at}}

\date{}

\maketitle


\renewcommand{\thefootnote}{}




\renewcommand{\thefootnote}{\arabic{footnote}}
\setcounter{footnote}{0}


\begin{abstract} 
The number of solutions
of the diophantine equation
$\sum_{i=1}^k \frac{1}{x_i}=1,$
in particular when the $x_i$ are
distinct odd positive integers is investigated.
The number of solutions $S(k)$ in this case is, for odd $k$:
\[\exp \left( \exp \left( c_1\, \frac{k}{\log k}\right)\right)
\leq S(k) 
\leq \exp \left( \exp \left(c_2\, k \right)\right)  \]
with some positive constants $c_1$ and $c_2$. This improves upon an earlier 
lower bound of $S(k) \geq \exp \left( (1+o(1))\frac{\log 2}{2} k^2\right)$.
\end{abstract}

\section{Introduction}
In this paper we study the number of solutions
of the diophantine equation
\begin{equation}{\label{eq:main}}
 \sum_{i=1}^k \frac{1}{x_i}=1,
\end{equation}
in particular, where the $x_i$ have some restrictions, such as all $x_i$ are
distinct odd positive integers. 
Let us first review what is known for distinct positive integers, 
without further restriction:
Let 
\[{\cal X}_k=\{(x_1, x_2, \ldots,x_k): \sum_{i=1}^k \frac{1}{x_i}=1, \quad
 0<x_1<x_2< \cdots < x_k\}.\]
It is known that 
\begin{equation}{\label{eq:bounds}}
\exp \left( \exp \left(((\log 2)(\log 3) +o(1))\frac{k}{\log k}\right)\right)
\leq |{\cal X}_k| \leq c_0^{(\frac{5}{3}+\varepsilon) \, 2^{k-3} },
\end{equation}
where $c_0=1.264\ldots$ is $\lim_{n \rightarrow \infty} u_n^{1/2^n}$,
$u_n=1, u_{n+1}=u_n (u_n+1)$.

The lower bound is due to Konyagin \cite{konyagin}, 
the upper bound due to Browning and
Elsholtz \cite{browning-elsholtz}.
Earlier results on the upper and lower bounds were 
due to S\'{a}ndor \cite{sandor}
and Erd\H{o}s, Graham and Straus (see \cite{erdosandgraham}, page 32).

The set of solutions has also been investigated with various restrictions
on the variables $x_i$.
A quite general and systematic investigation
of expansions of $\frac{a}{b}$ as a sum of unit fractions with restricted
denominators is due to Graham \cite{graham}.
Elsholtz, Heuberger, Prodinger \cite{elsholtz-heuberger-prodinger}
gave an asymptotic formula for the number of solutions of
({\ref{eq:main}}), with two main terms, when the $x_i$ are 
(not necessarily distinct) powers of a fixed integer $t$.

Another prominent case is when all denominators $x_i$ are odd.
Sierpi\'{n}ski \cite{sierpinski} proved that a nontrivial solution exists.
It is known 
that for $k=9$ there are exactly 5 solutions, and for $k=11$, there are exactly 
379,118 solutions (see \cite{shiu, arce-nazario-castro-figueroa}).
Chen, Elsholtz and Jiang \cite{chen-elsholtz-jiang}
showed that for odd denominators $x_i$
the number of solutions of (\ref{eq:main}) is 
increasing with a lower bound of $\sqrt{2}^{k^2(1+o(1))}$.
Other types of restrictions on the denominator have been studied, e.g. by
Croot \cite{croot} and Martin \cite{Martin}. The number of solutions of the equation
$\frac{m}{n}=\sum_{i=1}^k\frac{1}{x_i}$
have also been estimated by Elsholtz and Tao \cite{elsholtz-tao}.

In this paper we take inspiration from the proof of 
Chen et al.~\cite{chen-elsholtz-jiang} for odd denominators,
and the proof of Konyagin \cite{konyagin}
for lower bounds in the case of unrestricted $x_i$. As Konyagin's proof makes
crucial use of ingenious identities, involving a lot of even numbers, 
 it seems unclear whether one can generalize it to odd integers.
Here is our main result:

%

\begin{theorem}
Let $s\geq 1$ and let $\{p_1, \ldots, p_s\}$ denote a set of primes,
and let $P=p_1 \cdots p_s$ be squarefree.
Let $k$ be sufficiently large. Moreover, if $P$ is even, let $k$ be odd.
Let 
\[{\cal X}_{k,P}=\{(x_1, x_2, \ldots,x_k): \sum_{i=1}^k \frac{1}{x_i}=1, 
\text{ with distinct positive } x_i \equiv \pm 1\bmod P\}.\]
There is some positive constant $c(P)$ such that the following holds:
\[
 |{\cal X}_{k,P}| \geq 
\exp \left( \exp \left( c(P) \frac{k}{\log k}\right)\right). \]
\end{theorem}
The case $P=2$ is the case of odd denominators:
\begin{corollary}
Let $k$ be odd, and
\[{\cal X}_{k,\rm{odd}}=\{(x_1, x_2, \ldots,x_k): \sum_{i=1}^k \frac{1}{x_i}=1,
\text{ with odd distinct positive } x_i\}.\]
There is some positive constant $c$ such that the following holds:
\[
 |{\cal X}_{k,{\rm{odd}}}| \geq
\exp \left( \exp \left( c\, \frac{k}{\log k}\right)\right). \]
\end{corollary}
For comparison, an upper bound of type $\exp \left(\exp (c_2\, k)\right)$ 
follows from the unrestricted case, see ({\ref{eq:bounds}}).
\section{Proof}

\begin{lemma}{\label{lem:primitive}}
Let $P>1$ be a squarefree integer.
Let $\omega(n)$ denote the number of distinct prime factors of $n$, and $d(m)$
the number of divisors of $n$.
The following holds: $\omega(P^m - 1) \geq d(m)- 6$.
\end{lemma}
\begin{proof}
Due to a result of Bang, Zsigmondy, Birkhoff and Vandiver
(see e.g. Schinzel \cite{schinzel}),
it is known that for $n > 6$ the values of $P^n-1$ 
have at least one \emph{primitive} prime factor. (A prime factor of the
sequence $P^n-1$ is primitive if it divides $P^n-1$, but does not divide any
$P^m-1$ with $m<n$.).

Let $m=m_1m_2$. For each divisor $m_1$ one has the factorization
\[P^m-1=(P^{m_1}-1)(P^{m_1 m_2-m_1}+ P^{m_1 m_2-2m_1}+ \cdots +P^{m_1}+1),\]
hence the number of prime factors of $P^m-1$ is at least the sum of the number
of primitive prime factors of
$P^{m_1}-1$, for all possible divisors $m_1$ of $m$.
\end{proof}

%

\begin{lemma}{\label{lemma:wigert}}
For $X \geq  3$, there exists a natural number $m <X$ such that
$d(m) > \exp \left((\ln 2 + o(1))
\frac{\ln X}{\ln \ln X}\right)$ as $X \rightarrow \infty$.
\end{lemma}
This follows from a theorem of Wigert \cite{Wigert:1907}, 
but can also be seen directly.
Let $P_r= \prod_{i=1}^r q_i$ be the product over the first primes, and choose 
$m=P_r$, if $P_r \leq X < P_{r+1}$. Then $d(m)=2^r=
 \exp \left((\ln 2 + o(1))\frac{\ln m}{\ln \ln m}\right)=
 \exp \left((\ln 2 + o(1))
\frac{\ln X}{\ln \ln X}\right)$.
Taking the first $r$ odd primes, one can also find an odd number $m$ 
of this type.

\begin{lemma}{\label{lemma:albada-lint-progressions}}
For every $a,b,n_0\in \N$ the following holds:
every positive integer can be written as a finite sum of distinct fractions of
the form $\frac{1}{an+b},\, n \geq n_0$.
\end{lemma}

This result with $n_0=0$ was originally 
proved by van Albada and van Lint \cite{albada-lint}.
The result for general $n_0$ easily follows by using the 
progression $a'n+b'=an+(a n_0+b), n \geq 0$.

As an easy consequence we have:
\begin{lemma}{\label{lemma:consequence}}

There exist distinct
positive integers 
\[l_1, \ldots, l_{r_1}, m_1, \ldots,m_{r_2},n_1, \ldots, n_{r_3},\]
all larger than $1$,
in the residue class $1 \bmod 3P(P^2-1)$ such that the following holds:
\[
\sum_{i=1}^{r_1}\frac{1}{l_i}=P-2,\qquad 
\sum_{i=1}^{r_2}\frac{1}{m_i}=1, \qquad
\sum_{i=1}^{r_3}\frac{1}{n_i}=P, 
\]
\end{lemma}
If $P=2$, then $r_1=0$, otherwise $r_1,r_2,r_3>0$.
Moreover, it is clear that $r_2 \equiv 1 \bmod P$.

\begin{proof}[Proof of Theorem]

The idea employed in \cite{chen-elsholtz-jiang} and \cite{konyagin}
is to write $1$ as a sum of fractions where one denominator
has a large number of divisors, and to split this fraction recursively into
several fractions, where (at least) one of these has again a large number of
divisors.

Here we show that it is possible to have, for any given $t \in \N$, 
the fraction $\frac{1}{P^{t}-1}$ as one of these fractions.
Let us start with the trivial decomposition 
\[1=\frac{1}{P-1} + \frac{P-2}{P-1}.\]
In order to avoid that the denominator $P-1$ occurs more than once
we use Lemma \ref{lemma:consequence} to write the integer $P-2$
as a sum of distinct unit fractions, with $l_i>1$:
$P-2 =\sum_{i=1}^{r_1} \frac{1}{l_i}$.

Next we observe that any fraction $\frac{1}{P^n-1}$ can be decomposed
to obtain a sum of unit fractions containing
a) $\frac{1}{P^{2n}-1}$ or b) $\frac{1}{P^{n+1}-1}$.

\[(a)\quad  \frac{1}{P^n-1}=\frac{1}{P^n+1}+\frac{1}{P^{2n}-1}
+ \sum_{i=1}^{r_2} \frac{1}{(P^{2n}-1)m_i}.\]
By Lemma \ref{lemma:consequence}
\[ 1=\sum_{i=1}^{r_2} \frac{1}{m_i}, 
\quad m_i\equiv 1 \bmod 3P(P^2-1), m_i >1 \text{ and distinct}.\]
Note that all occurring denominators are distinct, with the possible exception
 that $P^n+1=P^{2n}-1$ holds if $P=2,n=1$. In this case,
one rewrites $\frac{1}{P+1}=\frac{1}{3}=\sum_{i=1}^{r_2} \frac{1}{3m_i}$. 
These denominators have not been used before, as the $l_i$ or $m_i$ are
congruent to $1 \bmod 3$, whereas the new denominators $3m_i$ are not.

\[ (b)\quad \frac{1}{P^n-1}=\frac{1}{P^{n+1}-1}+\frac{P-1}{(P^{n}-1)(P^{n+1}-1)}
+ \frac{P-1}{(P^{n+1}-1)}.\]
Note that these three fractions are unit fractions,
as the denominators are divisible by $P-1$.
These three fractions are distinct, 
unless $n=1$. In this case
the fraction $\frac{1}{P^2-1}$ occurs twice and one of these is
rewritten
as $\frac{1}{P^2-1}=\sum_{i=1}^{r_2} \frac{1}{(P^2-1)m_i}$.
These denominators have not been used before, as the previous denominators $l_i$ and $m_i$ were by 
construction congruent to $1 \bmod P^2-1$.
Also, $P^n+1, P^{2n}-1, (P^{2n}-1)m_i$ are new.

For constructing a solution with $\frac{1}{P^t-1}$ we write $t$ in 
binary. The first binary digit is of course $1$.
For the positions $i\geq 2$ 
we perform two different types of steps, corresponding to (a) and (b) above:\\
1) If the $i$-th 
leading position is a 0, then we take the ``doubling'' a).\\
2) If the $i$-th 
leading position is a 1, then we first take the doubling a), 
followed by an ``addition'' b),

For example, if $t=53=110101_2$
and starting from left to right:
\[
\begin{array}{rrrrrrrrr}
i=1|&&2|&3|&&4|&5|&&6|\\
1|&&1|&0|&&1|&0|&&1\\
|&a&b|&a|&a&b|&a|&a&b\\
n=1|&2&3|&6|&12&13|&26|&52&53\\
\end{array}\]

Generally, any integer $t$ can be obtained in at most $2\frac{\log t}{\log 2}$ such
steps a) or b).
In other words, starting from $n=1$ we can obtain a decomposition
\[1=\frac{1}{P^t-1}+\sum_{i=1}^{k'-1} \frac{1}{x_i}\] with
 $k'=O(r_1+r_2\log t+r_3)=O_P(\log t)$ unit fractions. 
Observe that all denominators have been
rearranged to be distinct.


We next come to the most crucial step, which determines the number of
solutions:

\begin{lemma}
Let $\sum_{i=1}^{r_3} \frac{1}{n_i}=P$ (by Lemma \ref{lemma:consequence}).
\begin{itemize}
\item[a)]
For any divisor $d|(P^t-1)$ the following is an identity.
\[\frac{1}{P^t-1}=\frac{1}{P^t-1+Pd}
+\sum_{i=1}^{r_3} \frac{1}{\frac{P^t-1}{d} (P^t-1+Pd)n_i}.\]
\item[b)]
The number of divisors $d|P^t-1$ with $
d\equiv 1 \bmod P$ is at least
$2^{\frac{\omega(P^t-1)}{P}}$.
\item[c)] If $d\equiv 1 \bmod P$, then {\emph{all}}
denominators are $\pm 1 \bmod P$.
\end{itemize}
\end{lemma}
Part a) and c) are easy to verify. For part b) observe:
For any $P$ prime factors $p_k$, being coprime to $P$,
there is at least one subset of these primes, whose product is
$1\bmod P$. Indeed, the sequence $a_1=p_1, a_2=p_1p_2, ..., a_{P}=\prod_{k=1}^{P} p_k$ must have 
two members $a_i, a_j$, say, which are equivalent modulo $P$. Then
$\frac{a_j}{a_i}= \prod_{k=i+1}^j p_k\equiv 1 \bmod P$.
Therefore, the number of divisors $d \equiv 1 \bmod P$ is at least
$2^{\frac{\omega(P^t-1)}{P}}$.
(Clearly, this argument can be refined (see e.g. \cite{Drmota-Skalba}), 
but this would not improve our final result.)
All solutions produced in this way are distinct, as each solution has
a unique denominator $P^t-1+Pd$. Moreover, as all these denominators are greater than
$P^t$, and as in our application $t$ will be chosen large, these new denominators are greater than
those that have been used before.

We choose $t$ as a product of the first primes. By Lemma {\ref{lemma:wigert}}
the number of divisors, and hence the number of solutions satisfies:

\[  |{\cal X}_{k,P}|\geq 2^{\omega(P^t-1)/P}\geq 
2^{(d(t)-6)/P}\geq 2^{\exp \left( \frac{\log 2 +o(1)}{P} \frac{\log t}{\log \log t}
  \right) }\geq 
 \exp\left( \exp (c(P)  k/\log k)\right).\]
Recall that the number of fractions is $k=O_P(\log t)$.

Finally let us comment on the condition that $k$ is odd, (see statemnet of the Theorem),
when $P$ is even.
By multiplying equation (\ref{eq:main}) by its common denominator, and
reducing modulo $P$ it is clear that this condition is necessary.
The condition is also sufficient as in view of step a) we can replace 
one fraction by $r_2+2$ fractions. Again, by the same argument 
$r_2 \equiv 1 \bmod P$, so that effectively we replace one fraction by 3 fractions (modulo $P$).
Iterating this, we can reach any residue class modulo $P$, 
when $P$ is odd, and the odd residue classes, when $P$ is even.
The number of extra fractions required is $O(P\, r_2)=O_P(1)$. This 
does not influence the overall result. In any case, the theorem is valid for sufficiently 
large $k\geq k_P$, with this necessary and sufficient congruence obstruction.

\end{proof}

\begin{rem}
We have not worked out the constant $c(P)$. One may observe that
$c(P)$ might be as small as $\frac{1}{r_2}$. To estimate $r_2$ one observes that
$\sum_{i=1}^{r_2} \frac{1}{i\, 3P(P^2-1)-1}\geq 
\sum_{i=1}^{r_2} \frac{1}{m_i}\approx \frac{\log r_2}{3P(P^2-1)} > P$ 
must hold.
Hence $r_2$ appears to be at least of exponential growth in $P$.
Taking denominators $x_i$ only coprime to $P$, but not necessarily restricted 
to $x_i \equiv \pm 1 \bmod 3P(P^2-1)$ would improve this constant $c(P)$.

\end{rem}
I would like to thank Sergei Konyagin for insightful comments on the problem
during a conference at CIRM (Luminy). 
The paper has been completed during a very pleasant stay 
at Forschungsinstitut Mathematik (FIM) at ETH Z\"urich.


\begin{thebibliography}{99}
\bibitem{albada-lint}
P.J. van Albada, J.H. van Lint,
Reciprocal bases for the integers.
 Amer. Math. Monthly 70 (1963), 170--174. 

\bibitem{arce-nazario-castro-figueroa}
R. Arce-Nazario, R. Castro, F. Figueroa, R.
 On the number of solutions of 
{$\sum_{i=1}^{11}\frac1{x_i}=1$} in distinct odd natural numbers,
J.~Number Theory 133 no. 6, (2013), 2036--2046.


\bibitem{browning-elsholtz}
T.~D. Browning and C.~Elsholtz.
The number of representations of rationals as a sum of unit
  fractions.
 Illinois J. Math. 55 (2011), no. 2, 685--696. 

\bibitem{chen-elsholtz-jiang}
Y.-G. Chen, C.~Elsholtz, and L.-L. Jiang.
Egyptian fractions with restrictions.
 Acta Arith. 154 (2012), no. 2, 109--123.

\bibitem{croot}
E.S. Croot, 
On a coloring conjecture about unit fractions. 
Ann. of Math. (2) 157 (2003), no. 2, 545--556.

\bibitem{Drmota-Skalba}
M. Drmota, M. Ska\l ba,
Equidistribution of divisors and representations by binary quadratic forms,
Int. J. Number Theory 9 (2013), 2011--2018. 


\bibitem{elsholtz-heuberger-prodinger}
C.~Elsholtz, C.~Heuberger, and H.~Prodinger.
The number of {H}uffman codes, compact trees, and sums of unit
  fractions.
 IEEE Trans. Inform. Theory 59 (2013), no. 2, 1065--1075.

\bibitem{elsholtz-tao}
C.~Elsholtz and T.~Tao.
\newblock Counting the number of solutions to the {E}rd{\H o}s-{S}traus
  equation on unit fractions.
 J. Aust. Math. Soc. 94 (2013), no. 1, 50--105. 


\bibitem{erdosandgraham}
P.~ Erd\H{o}s and R.L. Graham, 
Old and new problems and results in combinatorial number theory. 
Monographie No.28 de L'Enseignement Math\'{e}matique. Gen\`{e}ve, 
128 p. (1980). 

\bibitem{graham}
R. L. Graham, On finite sums of unit fractions, 
Proc. London Math. Soc. 14 (1964), 193--207.

\bibitem{Martin} G. Martin, Dense Egyptian fractions, Trans.
Amer. Math. Soc. 351 (1999), 3641--3657.


\bibitem{konyagin}
S. V. Konyagin, Double Exponential Lower Bound for the Number of 
Representations of Unity by Egyptian Fractions,
 Math. Notes 95 (2014), no. 1-2, 277--281. 


\bibitem{sandor}
C. S\'{a}ndor, On the number of solutions of the Diophantine equation 
$\sum\sp n\sb {i=1}\frac{1}{x\sb i}=1$. 
Period. Math. Hungar. 47 (2003), no. 1-2, 215--219.

\bibitem{schinzel}
A. Schinzel, On primitive prime factors of $a^n - b^n$, Proc. Camb. Phil.
Soc. 58 (1962), 555--562.

\bibitem{shiu}
P. Shiu,
Egyptian fraction representations of 1 with odd denominators,
Math. Gaz., 93 (2009), pp. 271--276.

\bibitem{sierpinski}
W. Sierpi\'{n}ski, Sur les d\'{e}compositions de nombres 
rationnels en fractions primaires, Mathesis 65 (1956), 16--32.

\bibitem{Wigert:1907} S.~Wigert, 
Sur l'ordre grandeur du nombre de diviseurs d'un entier. 
Ark. Mat. 3, no. 18 (1907), 1--9.



\end{thebibliography}
\end{document}